\documentclass[12pt, twoside]{article}
\usepackage{amssymb,amsmath,amsthm, mathrsfs,bbm,changes,color}

\theoremstyle{definition}
\newtheorem{thm}{Theorem}[section]

\newtheorem{lem}[thm]{Lemma}

\newtheorem*{xrem}{Remark}

\numberwithin{equation}{section}

\newcommand{\keywords}[1]{\noindent\emph{Keywords:}\enspace#1}

\definecolor{DarkBlue}{rgb}{0,0,.8}

\def\b{\boldsymbol}

\def\e{\varepsilon}
\def\sgn{\mbox{sgn}}

\begin{document}

\title{An Eulerian formulation of immersed interface method for membranes with tangential stretching}

\author{Lei Li\\
Department of Mathematics, Duke University, Durham, NC 27708\\
leili@math.duke.edu}
\date{}

\maketitle

\begin{abstract}
The forces generated by moving interfaces usually include the parts due to tangential stretching. We derive an evolution equation for the tangential stretching, which then forms the basis of an Eulerian formulation based on level set functions. The jump conditions are then derived using the level set and stretch functions. Compared with the traditional jump conditions under Lagrangian formulation,  the jump conditions under the Eulerian formulation are compact and clean. The work here makes possible a local level set formulation for immersed interface method to simulate membranes or vesicles where the tangential forces are present.

\end{abstract}

\section{Introduction}
The moving interface problems in viscous fluids make up an important class of fluid-structure interaction problems \cite{peskin02,sm11}.  Among the various methods for dealing with moving boundaries, two classes of methods have seen a lot of applications: the immersed boundary method (IBM) invented by Peskin \cite{peskin02} and the immersed interface method (IIM) by Leveque $\&$ Li  \cite{ll94}. The forces generated by the interfaces are usually singularly concentrated. In the immersed boundary method, the forces are smeared out from the interface to grid points using smoothened delta functions and then the usual solvers for fluid equations are applied. In the immersed interface method, the singular forces are built into the jump conditions of the velocity, pressure and their derivatives. The immersed boundary method is usually first order (\cite{peskin02, mori08}), while the immersed interface method is expected to be more accurate since the smoothened delta functions are avoided (\cite{ll97,lilai01, beale15}). In \cite{ll03}, a hybrid of the two methods was used in cases where the tangential forces are not significant, but the accuracy introduced by the immersed interface method may be reduced if the tangential forces are significant.  

If Lagrangian markers (i.e. material coordinates) are used to track the interface, interpolation and resampling are usually needed after a certain time. Meanwhile, in the stiff interface case, the implicit schemes are desired to evolve the interface to avoid some instabilities \cite{ll03}. In an Eulerian formulation, we have a fixed grid in space and only update the quantities on these fixed grid points, instead of tracking each material point. Regarding the issues mentioned, the Eulerian formulation behaves better.  The level set method (\cite{os88, sethian1999level, of01}) is an important tool for the Eulerian formulation of moving interfaces (\cite{chmo96, sso94, xz03}). For the immersed boundary method, Cottet and Maitre \cite{cm04, cm06} proposed an Eulerian formulation based on level set functions since they showed that the stretching can be kept in the gradient of the level set function. However, the reinitialization becomes an issue as we desire to preserve the stretch information. In \cite{lxs17}, a locally gradient preserving reinitialization was proposed to overcome this issue. For the immersed interface method, in \cite{lilai01, tlk10}, the level set method was proposed to solve some problems, but the Eulerian formulation for interfaces with tangential forces was not developed systematically. The idea in \cite{cm04} could be applied to immersed interface methods, but using the gradient of the level set to compute the stretching may cause inaccuracy for the jump conditions. On one side, when the stretching is too big or too small, the level set function becomes poor in capturing the location and on the other side, the stretching itself is obtained from finite difference.

In this work, we derive the equation for the stretching directly. The jump conditions then are derived under an Eulerian fashion based on the stretch and level set functions. The derived jump conditions are clean and compact. This then makes possible a local level set method for simulating the membranes where the tangential forces are significant using immersed interface method in a simpler way. The paper is organized as follows. In Section \ref{sec:setup}, we give a brief description of the setup of the problem and some preliminaries. In Section \ref{sec:stretch}, we derive the evolution equation for the stretch function. In Section \ref{sec:eulerian}, we formulate a closed system with a fully Eulerian fashion based on the level set and stretch functions. In Section \ref{sec:jump}, we derive the jump conditions for the derivatives of the velocity field and pressure up to second order under the Eulerian formulation. Lastly, we briefly discuss how the the immersed interface method under our Eulerian formulation may be implemented in Section \ref{sec:iim}.

\section{Setup and preliminaries}\label{sec:setup}

Consider the Navier-Stokes equations that describe a viscous fluid in domain $\Omega$, $
\rho(\b{u}_t+\b{u}\cdot\nabla\b{u})-\mu\Delta\b{u}+\nabla p=\b{f}$, and $\nabla\cdot\b{u}=0$. 
$\rho$ is the density, $\mu$ is the viscosity, $p$ is the pressure and $\b{u}$ is the velocity. In the case studied here, $\b{f}$ is the force generated by the immersed interface.  Scale the lengths by $L$, time by $T$, velocity by $U=L/T$ and the body force by $\rho L/T^2$. The Reynolds number in this case is $Re=\rho L^2/(\mu T)$. The dimensionless Navier-Stokes equations read
\begin{gather}\label{eq:momscaled}
\begin{array}{c}
\b{u}_t+\b{u}\cdot\nabla\b{u}+\nabla p=\frac{1}{Re}\Delta\b{u}+\b{f},\\
\nabla\cdot\b{u}=0.
\end{array}
\end{gather}

The immersed interface is assumed to be closed with codimension $1$.  We denote it by $\Gamma(t)$ and assume that it is described by $\xi\to \b{X}(\xi, t)$, where $\xi$ is a Lagrangian parameter. The material point is simply convected by the fluid velocity field $\b{u}$:
\begin{gather}
\b{X}_t=\b{u}(\b{X}, t).
\end{gather}

At $t=0$, we construct a level set function so that $\{x: \phi=0\}=\Gamma$, and $\phi>0$ inside $\Gamma$. If the level set function evolves according to
\begin{gather}\label{eq:levelset}
\phi_t+\b{u}\cdot\nabla\phi=0,
\end{gather}
then $\{x: \phi=0\}=\Gamma(t)$ for all $t>0$. 

In \cite{cm04}, it's shown that 
\begin{gather}
|\b{X}_{\xi}(\xi, t)|=\alpha(\xi)|\nabla\phi(\b{X}, t)|.
\end{gather}
If $\xi\in\mathbb{R}^2$ and $X\in\mathbb{R}^3$, $|\b{X}_{\xi}|=|\b{X}_{\xi_1}\times\b{X}_{\xi_2}|$. 
This result indicates that the magnitude of the gradient vector of the level set function evaluated at the interface may be used to track the tangential stretching $|\b{X}_{\xi}|$. However, this may not be accurate enough and not convenient, especially when we want to reinitialize the level set functions \cite{cm06} (One possible way for reinitialization has been proposed in \cite{lxs17}).  We instead find an equation for $\chi=|\b{X}_{\xi}|$ directly in Section \ref{sec:stretch}. 

We now introduce some notations and conventions that will be used.

 For $x\in\Gamma$, $\b{n}(x)$ is the inward normal of $\Gamma$ and $\b{n}=\nabla\phi/|\nabla\phi|$ is an extension of the normal vector. For a quantity $A$ and $x\in\Gamma$, define $A^+(x)=\lim_{s\to 0^+}A(x+s\b{n})$, $A^-(x)=\lim_{s\to 0^-}A(x+sn)$ and 
\begin{gather}\label{eq:spatialjump}
[A]=A^+-A^-.
\end{gather} 
 $\Gamma^{\e}$ is a banded domain that encloses $\Gamma$. The surface on the side where $\b{n}$ is pointing to is $\Gamma^+$ while $\Gamma^-$ is on the other side. The distance between $\Gamma^+$ and $\Gamma^-$ is made to be $\e$.  For example, if $\Gamma=\{(x, y, z): x^2+y^2+z^2=1\}$, we define $\Gamma^{+}=\{(x, y, z): x^2+y^2+z^2=(1-\e/2)^2\}$ and $\Gamma^{-}=\{(x, y, z): x^2+y^2+z^2=(1+\e/2)^2\}$. $\Gamma^{\e}$ is the shell corresponding to $1-\e/2\le\sqrt{x^2+y^2+z^2}\le 1+\e/2$.
 
 We define $\nabla \b{u}$ by
 \begin{gather}
 (\nabla\b{u})_{ij}=\partial_i u_j.
 \end{gather} 
 The dot product between $\nabla\b{u}$ and a vector is given by 
 \begin{gather}
 (\nabla\b{u}\cdot\b{v})_i=\partial_i u_j v_j, \\
 (\b{v}\cdot\nabla\b{u})_j=v_i\partial_iu_j,
 \end{gather}
 where repeated index means summation. We use $\b{a}\b{b}$ to mean the tensor produce between $\b{a}$ and $\b{b}$.

\section{Evolution of the streching}\label{sec:stretch}
In this section, we find the equation that the tangential stretching $\chi=|\b{X}_{\xi}|$ satisfies.  To do this, we first introduce several lemmas. Some of them appeared in \cite{cm04, laili01} but we still provide proofs in the new setting. 

The no-slip condition condition is satisfied at the interface and hence $\b{u}$ is continuous, or $[\b{u}]=0$.  Away from the interface, thanks to the viscosity, $\b{u}$ is smooth and we have the following:
\begin{lem}\label{ass:u}
$[\b{u}]=0$. $\b{u}$ is smooth at $x\notin\Gamma$ and $\nabla\b{u}$ has limits on both sides of $\Gamma$. 
\end{lem}
This claim immediately implies that $\b{u}$ is Lipschitz continuous on any bounded domain.  
\begin{lem} 
The velocity field for the incompressible Navier-Stokes equations \eqref{eq:momscaled} satisfies the following jump conditions across the interface $\Gamma$
\begin{gather}
[\b{u}_t]+\b{u}\cdot[\nabla\b{u}]=0,\\
[\nabla\b{u}]\cdot\b{n}=0.
\end{gather}
\end{lem}
\begin{proof}
By $[\b{u}]=0$, taking the derivative along the material trajectory yields the first condition  (See also \cite{laili01}).

To show the second condition, we start with $\nabla\cdot\b{u}=0$, which holds in the distribution sense (in the case that $\b{u}$ is continuous, it is equivalent to that $\nabla\cdot\b{u}=0$ on both sides of the interface).  Pick a test function $\zeta$ which is smooth in the whole domain. We have by Divergence Theorem that
 \begin{gather}\label{eq:jumpgradu}
 \oint_{\partial \Gamma^{\e}}\zeta\sum_i\partial_k u_i \tilde{n}_i d\sigma=\int_{\Gamma^{\e}}\nabla\cdot(\zeta\partial_k \b{u})dV
 =\int_{\Gamma^{\e}}\partial_k\b{u}\cdot\nabla\zeta dV=O(\e).
 \end{gather}
where $\tilde{\b{n}}$ is the outer normal of $\Gamma^{\e}$. $d\sigma$ is the surface element while $dV$ is the volume element. In the limit $\e\to 0$, on the inner side $\tilde{\b{n}}=\b{n}$ and on the outer side $\tilde{\b{n}}=-\b{n}$. It follows that
$$
\int_{\Gamma}\zeta[\partial_k\b{u}]\cdot\b{n} d\sigma=0.
$$
Since $\zeta$ is arbitrary, we must have $[\partial_k\b{u}]\cdot\b{n}=0$.
\end{proof}

\begin{xrem}
In general, $[\b{u}_t]=-\b{u}\cdot[\nabla\b{u}]\neq 0$. However, the interface is moving and $[\b{u}_t]\neq 0$ is true only at one instantaneous time. The velocity $\b{u}$ is still expected to be continuous.
\end{xrem}

\begin{lem}
Consider $\nabla\b{u}=\nabla\b{u}(x)$, $x\notin\Gamma$. In the 3D case, for any $\b{v}$ and $\b{w}$, we have
\begin{gather}
(\b{v}\cdot\nabla\b{u})\times\b{w}+\b{v}\times(\b{w}\cdot\nabla\b{u})=-\nabla\b{u}\cdot(\b{v}\times\b{w}).
\end{gather}
In the 2D case, letting $\hat{z}$ be perpendicular to the fluid plane,   for any $\b{v}$, it holds that 
\begin{gather}
\hat{z}\times(\b{v}\cdot\nabla\b{u})=-\nabla\b{u}\cdot(\hat{z}\times\b{v}).
\end{gather} 
\end{lem}
\begin{proof}
This proof follows closely from \cite{cm04}. Clearly, we only need to show the first result since the second one is a special case of the first one. Let $$
E=\frac{1}{2}(\nabla\b{u}+\nabla\b{u}^T), \ \ \ \ A=\frac{1}{2}(\nabla\b{u}-\nabla\b{u}^T).
$$
$E$ is symmetric and hence $E=\sum_i\lambda_i \b{e}_i\b{e}_i$ where $\{\b{e}_i\}$ forms an orthonormal basis. The incompressibility condition ensures $\sum_i\lambda_i=0$.
\begin{gather*}
(\b{v}\cdot E)\times\b{w}+\b{v}\times(\b{w}\cdot E)
=\sum_i\lambda_i v_i\b{e}_i\times\b{w}+\sum_i\lambda_i w_i\b{v}\times\b{e}_i\\
=\sum_{i,m,n}(\lambda_i+\lambda_m)\e_{imn}v_iw_m\b{e}_n
=-\sum_{imn}\lambda_n\e_{imn}v_iw_n\b{e}_n=-E\cdot(\b{v}\times\b{w}).
\end{gather*}
$\e_{imn}$ is nonzero only if $i,m,n$ are not equal and thus $\e_{imn}(\lambda_i+\lambda_m)=-\e_{imn}\lambda_n$.

For the antisymmetric part, we can find $\Omega$ such that $A\cdot\b{q}=\Omega\times\b{q}$.
\begin{multline*}
(\b{v}\cdot A)\times\b{w}+\b{v}\times(\b{w}\cdot A)=-(A\cdot\b{v})\times\b{w}-\b{v}\times(A\cdot\b{w})\\
=-(\Omega\times\b{v})\times\b{w}-\b{v}\times(\b{\Omega}\times\b{w})
=-(\Omega\cdot\b{w})\b{v}+(\b{v}\cdot\b{\Omega})\b{w}\\
=-\Omega\times(\b{v}\times\b{w})=-A\cdot(\b{v}\times\b{w}).
\end{multline*}
This then ends the proof.
\end{proof}

Let $D/Dt=\partial_t+\b{u}\cdot\nabla$ denote the material derivative. The evolution equation of the tangential stretching is given by the following theorem:
\begin{thm}
Let $\chi=|\b{X}_{\xi}|$ (recall in 3D case, this means $|\b{X}_{\xi_1}\times\b{X}_{\xi_2}|$). It satisfies the following equation:
\begin{gather}
\frac{D}{Dt}\chi(\xi, t)=-(\b{n}\cdot\nabla\b{u}\cdot\b{n})\chi(\xi, t)=\big((I-\b{n}\b{n}):\nabla\b{u}\big)\chi(\xi, t),
\end{gather}
where $\b{n}\cdot\nabla\b{u}\cdot\b{n}$ is well-defined at the interface. 
\end{thm}

\begin{proof}
In the 2D case, $\xi\in\mathbb{R}$. Let $\hat{z}$ be a unit vector perpendicular with the fluid plane. $\chi=|\hat{z}\times\b{X}_{\xi}|$.
We have
\begin{gather*}
\frac{D}{Dt}\chi=\b{n}\cdot\Big(\hat{z}\times\frac{\partial}{\partial\xi}\b{u}(X(\xi,t),t)\Big).
\end{gather*}
Since $\nabla\b{u}$ is not continuous across the interface, we can not take the derivative using chain rule directly. Instead, we consider
the divided difference
\begin{gather*}
D_h^{\pm}\b{u}=\frac{\b{u}(X(\xi+h)\pm h^2 \b{n}(\xi+h))-\b{u}(X(\xi)\pm h^2\b{n}(\xi))}{h},
\end{gather*}
where $t$ dependence is ignored for convenience. 
Then, since $\b{u}$ is Lipschitz, as $h\to0$,
\begin{multline*}
\hat{z}\times D_h^{\pm}\b{u}=\hat{z}\times((\b{X}_{\xi}\pm h \b{n}_{\xi})\cdot\nabla \b{u}(X(\xi)\pm h^2 \b{n}(\xi)))+O(h)\\
=-\nabla\b{u}(X(\xi)\pm h^2 n(\xi))\cdot(\hat{z}\times(\b{X}_{\xi}\pm h^2 n_{\xi}))+O(h)\to -\nabla\b{u}(X)^{\pm}\cdot(\hat{z}\times\b{X}_{\xi}).
\end{multline*}
The two limits are equal to $-(\nabla\b{u}\cdot\b{n})\chi$ since $\hat{z}\times\b{X}_{\xi}$ is parallel with $\b{n}$ and $[\nabla\b{u}]\cdot\b{n}=0$.

On the other hand,
\begin{gather*}
D_h\b{u}=\frac{\b{u}(X(\xi+h))-\b{u}(\xi)}{h}+O(h),
\end{gather*}
since $\b{u}$ is Lipschitz continuous. This then means $\b{u}(\b{X}(\xi))$ is differentiable with respect to $\xi$ and $$
\hat{z}\times\partial_{\xi}\b{u}(X)=-\nabla\b{u}\cdot(\hat{z}\times\b{X}_{\xi})=-(\nabla\b{u}\cdot\b{n})\chi.
$$
Hence,
\begin{gather*}
\frac{D}{Dt}\chi=-\b{n}\cdot\nabla\b{u}\cdot(\hat{z}\times\b{X}_{\xi})=-(\b{n}\cdot\nabla\b{u}\cdot\b{n})\chi.
\end{gather*}
For the 3D case, the proof is similar. We use $\chi=|\b{X}_{\xi_1}\times\b{X}_{\xi_2}|$ and note $\b{X}_{\xi_1}\times\b{X}_{\xi_2}$ is parallel to $\b{n}$. The last inequality in the theorem follows from the fact that $I:\nabla\b{u}=\nabla\cdot\b{u}=0$.
\end{proof}

\section{An Eulerian formulation} \label{sec:eulerian}
Suppose that the membrane is described by the level set function $\phi$ which satisfies \eqref{eq:levelset} and that $\phi>0$ inside $\Gamma$. We now derive a closed system for simulating the moving interface with a fully Eulerian fashion.

\subsection{The elastic forces}
Extend the stretch function $\chi$ to a neighborhood of the interface  (we provide one possible extension later). Let $\b{n}=\nabla\phi/|\nabla\phi|$, which is a natural extension of the normal vector of the interface. We now derive the forces in terms of $\phi$ and $\chi$. 

Consider an energy that can be decomposed into two parts $\mathscr{E}=\mathscr{E}_s+\mathscr{E}_b$. $\mathscr{E}_s$ only depends on the stretching and $\mathscr{E}_b$ only depends on the curvature (see \cite{mmpr12}). 

The stretching part can be written as $\mathscr{E}_s=\int_{\Gamma} E_s(|\b{X}_{\xi}|)d\xi$ (see \cite{cm04}). In the 3D fluid domain case, $d\xi$ is the area element in the $\xi_1$-$\xi_2$ plane. 
Let $\delta(u)$ be the 1D Dirac delta function in variable $u$. Using the fact that $\delta(\phi)|\nabla\phi|$ is the delta function for variable $x$, the substitution $d\xi\to \delta(\phi)|\nabla\phi|dV/\chi$ could be applied. (Note that even though $\nabla\b{u}$ may be discontinuous across the interface, $\phi$ can be shown be to $C^1$ by investigating the equation for $\nabla\phi$ and $|\nabla\phi|$ at the interface is well defined.) This energy can be written as
\begin{gather}\label{eq:elasenergy}
\mathscr{E}_s=\int_{\Omega} \frac{1}{\chi}E_s(\chi)|\nabla\phi|\delta(\phi)dV
\end{gather}
Recall that $\Omega$ is the fluid domain.

The force is determined through
\begin{gather}\label{eq:elasf}
\frac{d}{dt}\mathscr{E}_s=-\int_{\Omega}\b{f}_s\cdot\b{u}dV.
\end{gather}
There is a total gradient that can't be determined since $\nabla\cdot\b{u}=0$, but it can be absorbed into the pressure. After a tedious derivation, we obtain:
\begin{gather}
\b{f}_s=\nabla E_s'(\chi)\cdot(I-\b{n}\b{n})|\nabla\phi|\delta(\phi)+E_s'(\chi)\kappa \nabla\phi \delta(\phi),
\end{gather}
where 
\begin{gather}
\kappa=-\nabla\cdot(\nabla\phi/|\nabla\phi|)=-\nabla\cdot\b{n}
\end{gather} is the curvature chosen so that the curvature of a convex surface is positive. $I$ is the identity tensor and $\b{n}\b{n}$ is the tensor product between $\b{n}$ and itself. We provide the derivation in Appendix \ref{app:stretchf}.
\begin{xrem}
Note that $\nabla\cdot\b{n}$ is well defined at the interface. Since $\b{n}$ is a unit vector field, $((\nabla\b{n})\cdot\b{n})^{\pm}=0$. Then, the $\b{n}\b{n}$ component of $\nabla\b{n}$ is zero. $\nabla\cdot\b{n}$ is determined by $\b{e}_i\b{e}_i$ components only where $\b{e}_i$ is a tangent vector. These components are only given by the values of $\b{n}$ on $\Gamma$.  
\end{xrem}

The energy that depends on the curvature is given by $\mathscr{E}_{b}=\int_{\Gamma(t)}E_b(\kappa)d\sigma$.  For convenience, we have assumed that  $E_b$ just depends on the curvature. Strictly speaking, after the elastic structure has been stretched, the bending stiffness may be changed and $E_{b}$ may not just depend on $\kappa$.  In the 2D fluid domain case $d\sigma$ is the arclength element and in the 3D fluid domain case, $d\sigma$ is the surface area element on $\Gamma$.
 
Since the Jacobian $d\sigma/d\xi=|X_{\xi}|=\chi$, we therefore could do the substitution $d\sigma\to \delta(\phi)|\nabla\phi|dV$ and find
\begin{gather}
\mathscr{E}_{b}=\int_{\Omega} E_b(\kappa)|\nabla\phi|\delta(\phi)dV,
\end{gather}
The principle of virtual work then gives (see \cite{mmpr12})
\begin{gather}
\b{f}_b=\nabla\cdot\left(-E_b(\kappa)\b{n}
-\frac{(I-\b{n}\b{n})}{|\nabla\phi|}
\cdot\nabla(E_b'(\kappa)|\nabla\phi|)\right)\nabla\phi \delta(\phi).
\end{gather}
Note that in \cite{mmpr12}, the normal is pointing outward and the curvature is $\nabla\cdot\b{n}$. The second term above is different from theirs with a negative sign. This force can be simplified to: 
\begin{gather}
\b{f}_b=\Big(E_b\kappa-E_b'\nabla\b{n}:(\nabla\b{n})^T-\Delta_{tan}E_b'(\kappa)\Big)\nabla\phi \delta(\phi)
\end{gather}
where 
\begin{gather}
\Delta_{tan}=((I-\b{n}\b{n})\cdot\nabla)\cdot ((I-\b{n}\b{n})\cdot\nabla)
\end{gather}
$\nabla\b{n}:(\nabla\b{n})^T=((I-\b{n}\b{n})\cdot\nabla\b{n}):((I-\b{n}\b{n})\cdot\nabla\b{n})^T$. $\b{f}_b$ is given by the shape of the interface only. We provide the computation for showing these two $\b{f}_b$ expressions are equivalent in Appendix \ref{app:curvf}.

The body force is then $\b{f}=\b{f}_s+\b{f}_b$. $\b{f}_b$ has only normal component while $\b{f}_s$ is related to the tension, and it has both tangential and normal components. The tangential component is due to the varying tension along the interface, which can also be significant in applications. 

 According to the forces given, only the zero level set $\Gamma$ and the values of $\chi$ on $\Gamma$ matter, which are determined by $\b{u}$ at the interface totally. Hence, to get the correct forces, we don't have to use the physical velocity for $\phi$ and $\chi$. We can use two arbitrary velocity fields $\b{v}$ and $\b{w}$ such that they are equal to the physical velocity field $\b{u}$ on $\Gamma$, but can result smooth $\b{f}_1$, $f_2$ and $\phi$ near $\Gamma$. One example is $\b{v}=\b{w}=T\b{u}$ with $T\b{u}$ being obtained by extending $\b{u}$ along $\nabla\phi$. Hence, we find an Eulerian formulation as
\begin{gather}\label{eq:ns}
\phi_t+\b{v}\cdot\nabla\phi=0,\\
\chi_t+\b{w}\cdot\nabla\chi=\Big((I-\b{n}\b{n}):\nabla\b{w}\Big)\chi,\ \ \  \b{n}=\nabla\phi/|\nabla\phi|, \\
\b{f}=\b{f}_1(x)|\nabla\phi|\delta(\phi)
+f_2(x)\nabla\phi\delta(\phi),\\
\b{f}_1=(I-\b{n}\b{n})\cdot\nabla E_s'(\chi),\\
f_2=\kappa E_s'(\chi)+E_b\kappa-E_b'\nabla\b{n}:(\nabla\b{n})^T-\Delta_{tan}E_b'(\kappa).
\end{gather}
This formulation can be simulated by the immersed boundary method, but the accuracy is not good as the discrete delta functions must be used. To achieve a better accuracy, one may use the immersed interface method and the singular forces are built into the jump conditions of the physical quantities so that the delta functions are avoided totally. 
\begin{xrem}
It is not hard to show that $\phi$ is $C^1$ and $\chi$ is continuous with $\b{v}=\b{w}=\b{u}$. Whether or not $\b{u}$ yields higher regularity remains an interesting question, which we won't study here. 
\end{xrem}

\section{Jump conditions}\label{sec:jump}
As mentioned before, in the immersed interface method setting, we build the effect of the immersed interface into jump conditions of the physical quantities to avoid using delta functions. For the jump conditions under Lagrangian fashion, one can refer to \cite{laili01, xw06}. The derived jump conditions are complicated for implementation. Due to the complexity, in \cite{ll03}, the authors proposed to deal with the tangential forces using the smoothened delta functions as in the immersed boundary method but build the normal component into jump conditions as in the immersed interface method. When the tangential forces are significant, this treatment is not excellent any more. We now derive the jump conditions under the Eulerian fashion. We will see that the tangential components can also conveniently be built into the jump conditions and the obtained expressions are clean and compact. 

We first of all derive the spatial jumps. The temporal jumps are then derived from the spatial jumps directly. Recall that the spatial jump is defined by Equation \eqref{eq:spatialjump}. Note that the derivation in this section actually applies for any forces $\b{f}_1$ and $f_2$, not just the elastic forces given in the previous section.
\subsection{Spatial jump conditions} 
\begin{thm}\label{thm:p}
With the setting given in Section \ref{sec:eulerian}, the jump conditions of $p$ and $\b{n}\cdot\nabla\b{u}$ are given by:
\begin{gather}
[p]=f_2,\\
\frac{1}{Re}\b{n}\cdot[\nabla\b{u}]=-\b{f}_1.
\end{gather}
\end{thm}
\begin{proof}
Multiply a test function $\zeta(x)$ in the Navier-Stokes equations \eqref{eq:momscaled} and integrate over $\Gamma^{\e}$. The inertial terms become $$
\int_{\Gamma^{\e}(t)}\zeta(\b{u}_t+\b{u}\cdot\nabla\b{u})dV=\frac{d}{dt}\int_{\Gamma^{\e}(t)}\zeta \b{u}dV=O(\e).
$$
As $\e\to 0$, since $p$ is bounded, the pressure term has the following limit:
\begin{gather*}
\int_{\Gamma^{\e}(t)}\zeta\nabla p dV=\int_{\partial\Gamma^{\e}}\tilde{\b{n}}\zeta p d\sigma-\int_{\Gamma^{\e}}p\nabla\zeta dV
=\int_{\Gamma^+}\b{n}\zeta p d\sigma-\int_{\Gamma^-}\b{n}\zeta p d\sigma+O(\e)\to \int_{\Gamma}\b{n}\zeta[p]d\sigma.
\end{gather*}
Similarly, $\nabla\b{u}$ is bounded and the diffusion term has the limit:
\begin{gather*}
\frac{1}{Re}\int_{\Gamma^{\e}}\zeta\Delta\b{u}dV\to \frac{1}{Re}\int_{\Gamma}\zeta\b{n}\cdot[\nabla\b{u}] d\sigma.
\end{gather*}

The following identity, as given in \cite{cm06},
\begin{gather}\label{eq:integrallevel}
\int_{|\phi|<\eta}g(x)dV=\int_{-\eta}^{\eta}\int_{\phi=\nu}g(x)|\nabla\phi|^{-1}d\sigma d\nu
\end{gather}
yields
\begin{gather*}
\int_{\Gamma^{\e}}\zeta\b{f}dV=\int_{\Gamma}\zeta[\b{f}_1+f_2(x)\b{n}]d\sigma.
\end{gather*}
Combining the results gives
\begin{gather*}
\b{n}[p]=\frac{1}{Re}\b{n}\cdot[\nabla\b{u}]+\b{f}_1+f_2\b{n}.
\end{gather*}
Since $[\nabla\b{u}]\cdot\b{n}=0$, $\b{n}\cdot[\nabla\b{u}]$ is then perpendicular with $\b{n}$. The results follow by comparing the components parallel and perpendicular with $\b{n}$.
\end{proof}

To derive $[\b{n}\cdot\nabla p]$, we'll basically follow \cite{laili01}, but in the current Eulerian setting, we need the following lemma about surface divergence:
\begin{lem}\label{lmm:surdiv}
Let $\b{v}$ be a smooth vector field defined in a neighborhood of the closed surface $\Gamma$. $\b{n}$ is some smooth extension of the normal vector such that $|\b{n}|=1$ in a neighborhood of $\Gamma$. Then, it holds that
\begin{gather}
\Big((I-\b{n}\b{n})\cdot\nabla\Big)\cdot\b{v}=\b{n}\cdot\nabla\times(\b{n}\times\b{v})+(\b{v}\cdot\b{n})(\nabla\cdot\b{n}),
\end{gather}
and it follows that
\begin{gather}
\int_{\Gamma}[(I-\b{n}\b{n})\cdot\nabla]\cdot\b{v}d\sigma
=\int_{\Gamma}(\b{v}\cdot\b{n})\nabla\cdot\b{n}d\sigma.
\end{gather}
\end{lem}
\begin{proof} 
The first equation is just direct computation where $(\partial_in_j)n_j=0$ is used. Integrating the first identity over the closed surface $\Gamma$ then yields the integral identity by Stokes theorem. 
\end{proof}

Define the projection of the gradient operator:
\begin{gather}
\nabla_{\Gamma}=(I-\b{n}\b{n})\cdot\nabla.
\end{gather}
\begin{lem}\label{lmm:normalgradp}
With the setting in Section \ref{sec:eulerian}, the jump condition of $\b{n}\cdot[\nabla p]$ is given by:
\begin{gather}
[\frac{\partial p}{\partial n}]=\nabla_{\Gamma}\cdot\b{f}_1=((I-\b{n}\b{n})\cdot\nabla)\cdot\b{f}_1.
\end{gather}
\end{lem}
\begin{xrem}
Note that $\nabla p$ has a delta-like singularity at the interface since $p$ is discontinuous across the interface. We can decompose $\nabla p$ into a singular distribution at the interface and a piecewise smooth vector field. $[\nabla p]$ is given by the jump of the piecewise smooth vector field and has nothing to do with the singular distribution.
\end{xrem}

\begin{proof}
Taking the divergence of the first equation in \eqref{eq:momscaled} yields,
$$
\nabla\cdot(\b{u}_t+\b{u}\cdot\nabla\b{u})+\Delta p=\nabla\cdot\b{f}.
$$
We multiply a test function $\zeta$ and integrate over $\Gamma^{\e}$. Integrating by parts and using the jump condition $[\b{u}_t]+\b{u}\cdot[\nabla\b{u}]=0$, we find the first term to satisfy $$
\int_{\Gamma^{\e}}\zeta\nabla\cdot(\b{u}_t+\b{u}\cdot\nabla\b{u})dV\to 0,\text{as}\ \e\to 0.
$$

The second term satisfies: 
\begin{multline*}
\int_{\Gamma^{\e}}\zeta\Delta p dV=\int_{\partial\Gamma^{\e}}\zeta\tilde{\b{n}}\cdot \nabla pd\sigma
-\int_{\partial\Gamma^{\e}}p\tilde{\b{n}}\cdot\nabla\zeta d\sigma+\int_{\Gamma^{\e}}p\Delta\zeta dV\\
\to
\int_{\Gamma}\zeta [\b{n}\cdot\nabla p]d\sigma-\int_{\Gamma}[p]\b{n}\cdot\nabla\zeta d\sigma.
\end{multline*}
The last term is 
\begin{gather*}
\int_{\Gamma^{\e}}\zeta\nabla\cdot\b{f}dV=-\int_{\Gamma^{\e}}\b{f}\cdot\nabla\zeta
=-\int_{\Gamma}(\b{f}_1+f_2\b{n})\cdot\nabla\zeta d\sigma,
\end{gather*}
by the formula \eqref{eq:integrallevel} again. 

Applying Lemma \ref{lmm:surdiv} on $\b{v}=\zeta P\b{f}_1=\zeta(I-\b{n}\b{n})\cdot\b{f}_1$, we have 
\begin{gather*}
-\int_{\Gamma}\b{f}_1\cdot\nabla\zeta d\sigma=-\int_{\Gamma}\b{f}_1\cdot(I-\b{n}\b{n})\cdot\nabla\zeta d\sigma=\int_{\Gamma}\zeta\nabla_{\Gamma}\cdot\b{f}_1 d\sigma.
\end{gather*}
The claim then follows by comparing the terms and the fact that $\zeta$ is arbitrary.
\end{proof}

The following lemma allows us to compute the tangential components of $[\nabla p]$ and $[\nabla\b{u}]$:
\begin{lem}\label{lmm:tangent}
Suppose $A$ is some quantity defined in a neighborhood of $\Gamma$, then
\begin{gather}
(I-\b{n}\b{n})\cdot[\nabla A]=\nabla_{\Gamma}[A]=(I-\b{n}\b{n})\cdot\nabla[A].
\end{gather}
\end{lem}
\begin{proof}
Let $\gamma$ be a path in $\Gamma$ with endpoints being $x_0$ and $x_1$. Let $T_{\e}^+(x)=x+\frac{\e}{2}\b{n}$ and
$T_{\e}^-(x)=x-\frac{\e}{2}\b{n}$. Then, $A(T_{\e}^{\pm}(x_1))-A(T_{\e}^{\pm}(x_0))=\int_{T_{\e}^{\pm}(\gamma)}d\b{l}\cdot\nabla A$. Hence, 
\begin{gather*}
\int_{\gamma}d\b{l}\cdot[\nabla A]=\lim_{\e\to 0}\Big(\int_{T_{\e}^{+}(\gamma)}d\b{l}\cdot\nabla A-\int_{T_{\e}^{-}(\gamma)}d\b{l}\cdot\nabla A\Big)=[A](x_1)-[A](x_0)=\int_{\gamma}d\b{l}\cdot\nabla_{\Gamma}[ A].
\end{gather*}
Since $\gamma$ is arbitrary, the claim then follows.
\end{proof}
\begin{thm}
With the setting in Section \ref{sec:eulerian}, 
\begin{gather}
\frac{1}{Re}[\partial_i\b{u}]=-n_i\b{f}_1.\\
[\partial_i p]=(\partial_i-n_i\b{n}\cdot\nabla)f_2+n_i\nabla_{\Gamma}\cdot\b{f}_1.
\end{gather}
where $\partial_i$ denotes the derivative on the $i$-th Cartesian coordinate, namely $x,y$ or $z$.
\end{thm}
\begin{proof}
Applying Lemma \ref{lmm:tangent} to $\b{u}$ and $p$, we find $(I-\b{n}\b{n})\cdot[\nabla\b{u}]=0$
and that $(I-\b{n}\b{n})\cdot[\nabla p]=(I-\b{n}\b{n})\cdot\nabla f_2$. Hence, combing Theorem \ref{thm:p} and  Lemma \ref{lmm:normalgradp}, we  have
$\frac{1}{Re}[\nabla\b{u}]=-\b{n}\b{f}_1$
and $[\nabla p]=(I-\b{n}\b{n})\cdot\nabla f_2+\b{n}\nabla_{\Gamma}\cdot\b{f}_1$.
\end{proof}

We move on to the jump conditions of $\nabla\nabla\b{u}$, the Hessian of the velocity field. We start with showing a lemma regarding the symmetry of $\nabla\b{n}$:
\begin{lem}\label{lmm:sym}
In the 3D space, fix a point $x$ on a smooth surface $\Gamma$. Suppose $\b{t}$ and $\b{b}$ are two unit vectors in the tangent plane at $x$ and $\b{t}\cdot\b{b}=0$, $\b{t}\times\b{b}=\b{n}$. Then, for any smooth extension of the normal vector $\b{n}$ into a neighborhood of $x$, the following holds at $x$
\begin{gather}
\b{t}\cdot\nabla\b{n}\cdot\b{b}=\b{b}\cdot\nabla\b{n}\cdot\b{t}.
\end{gather}
\end{lem}
\begin{proof} 
The claim simply says $
\b{t}\cdot\nabla\b{n}\cdot\b{b}-\b{t}\cdot\nabla\b{n}\cdot\b{b}
=\b{t}\cdot(\nabla\b{n}-\nabla\b{n}^T)\cdot\b{b}=0$, or equivalently
$\b{b}\cdot((\nabla\times\b{n})\times\b{t})=-\b{n}\cdot(\nabla\times\b{n})=0$.
Note that $\b{t}\cdot\nabla\b{n}$ and $\b{b}\cdot\nabla\b{n}$ only depend on the values of $\b{n}$ on the surface, which is independent of the extension of $\b{n}$. Hence, the value $\b{n}\cdot(\nabla\times\b{n})$ is independent of the extension of $\b{n}$. Picking the special extension $\b{n}=\nabla\phi/|\nabla\phi|$ yields $\b{n}\cdot(\nabla\times\b{n})=0$.
\end{proof}
\begin{xrem}
By this result, we can obtain the following amusing claim, though not relevant to us:

Suppose $\b{v}$ is a smooth vector field defined in a domain (open, connected) in $\mathbb{R}^3$ and $\b{v}(x_0)\neq0$. If $\b{v}\cdot(\nabla\times\b{v})\neq 0$ at $x_0$, then, there is no neighborhood $U$ of $x_0$ and smooth surface $S$ such that $S$ passes through point $x_0$ and $\b{v}$ is perpendicular with the surface at every point $x\in U\cap S$.
\begin{proof}
Suppose there is $S$. Then, $\b{n}=\b{v}/|\b{v}|$ is an extension of the normal vector of $S$. Then, $\b{n}\cdot(\nabla\times\b{n})=0$, which is equivalent to $\b{v}\cdot(\nabla\times\b{v})=0$, but this is not true.
\end{proof}
\end{xrem}

We now prove the following jump conditions for the Hessian of $\b{u}$. Below, Einstein's convention is used, namely repeated indices imply summation. $\delta_{ij}$ is the usual Kronecker delta. 
\begin{thm}
With the setting in Section \ref{sec:eulerian}, the jump conditions of $\partial_i\partial_j\b{u}$ are given by
\begin{multline}
\frac{1}{Re}[\partial_i\partial_j\b{u}]=
-\Big(\partial_jn_i-n_k(\partial_kn_i)n_j\Big) \b{f}_1\\
-(n_j\partial_i+n_i\partial_j-2n_in_jn_k\partial_k)\b{f}_1
+n_in_j([\nabla p]-\kappa \b{f}_1),
\end{multline}
where $i$ or $j$ means the component of the Cartesian coordinates.  In the 2D case, 
$\partial_in_j=-\kappa(\delta_{ij}-n_in_j)+n_in_k\partial_kn_j$,
so that 
\begin{gather}
\frac{1}{Re}[\partial_i\partial_j\b{u}]=
\kappa(\delta_{ij}-2n_in_j) \b{f}_1-(n_j\partial_i+n_i\partial_j-2n_in_jn_k\partial_k)\b{f}_1+n_in_j[\nabla p].
\end{gather}
\end{thm}

\begin{xrem}
Like $\nabla p$, $\nabla\nabla\b{u}$ has a delta-like singularity at the interface since the force has a such singularity,  but the jump of $\nabla\nabla\b{u}$ is well-defined and is unrelated to this singularity. 
\end{xrem}
 
 \begin{proof}
Applying Lemma \ref{lmm:tangent} to $[\nabla\b{u}]$ directly, we have
\begin{gather*}
\frac{1}{Re}(I-\b{n}\b{n})\cdot[\nabla\nabla\b{u}]=(I-\b{n}\b{n})\cdot Q,
\end{gather*}
where $Q$ is a third order tensor given by $$
Q_{ij:}=-\partial_in_j \b{f}_1
-n_j\partial_i\b{f}_1.
$$
Since $[\b{u}_t]+\b{u}\cdot[\nabla\b{u}]=0$, we have 
\begin{gather*}
[\nabla p]=\frac{1}{Re}[\Delta \b{u}].
\end{gather*}

Let's now continue to prove for 3D case. The proof for the 2D case is much easier but follows a similar argument and thus we choose to omit it.

Let $\b{t}$ and $\b{b}$ be in the tangent plane and $\{\b{t}, \b{b}, \b{n}\}$ forms an orthonormal basis.  $I=\b{t}\b{t}+\b{n}\b{n}
+\b{b}\b{b}$ where $I$ is the identity tensor. (Recall that $\b{v}\b{w}$ means the tensor product between $\b{v}$ and $\b{w}$.) In index form,
\begin{gather}\label{eq:identity}
t_it_j+b_ib_j+n_in_j=\delta_{ij}.
\end{gather}

We denote $M:N=M_{ij}N_{ij}$ and that $\partial_{tb}=\b{t}\b{b}:\nabla\nabla$ (Note that this is not $\b{t}\cdot\nabla(\b{b}\cdot\nabla)$). Then, the following hold
\begin{gather*}
\frac{1}{Re}[\partial_{tt}\b{u}]=\b{t}\b{t}:Q=-t_i\partial_in_jt_j \b{f}_1,\\
\frac{1}{Re}[\partial_{bb}\b{u}]=\b{b}\b{b}:Q=-b_i\partial_in_jb_j\b{f}_1.
\end{gather*}
Since $\Delta=\partial_{tt}+\partial_{bb}+\partial_{nn}$, we have
\begin{gather*}
\frac{1}{Re}[\partial_{nn}\b{u}]=[\nabla p]-(I-\b{n}\b{n}):Q=[\nabla p]-(-\nabla\cdot\b{n})\b{f}_1=[\nabla p]-\kappa \b{f}_1.
\end{gather*}
Clearly,
\begin{gather*}
\frac{1}{Re}[\partial_{tn}\b{u}]=-\b{t}\cdot\nabla\b{f}_1,\\
\frac{1}{Re}[\partial_{bn}\b{u}]=-\b{b}\cdot\nabla\b{f}_1,\\
\frac{1}{Re}[\partial_{tb}\b{u}]=-(\b{t}\b{b}:\nabla\b{n})\b{f}_1.
\end{gather*}
Since $\partial_{tb}\b{u}=\partial_{bt}\b{u}$ for the derivatives, we must have $\b{t}\b{b}:\nabla\b{n}=\b{b}\b{t}:\nabla\b{n}$ to be consistent. This is exactly the result in Lemma \ref{lmm:sym}. We denote $(\b{a}\b{b}\b{c}\b{d}:\nabla\b{n})_{ij}=a_ib_j((\b{c}\cdot\nabla)\b{n})\cdot\b{d}$.
Then, we find the jump of the Hessian
\begin{multline*}
\frac{1}{Re}[\nabla\nabla\b{u}]-\b{n}\b{n}\frac{1}{Re}[\partial_{nn}\b{u}]
=-\b{t}\b{t}\b{t}\b{t}:\nabla\b{n}\b{f}_1-\b{b}\b{b}\b{b}\b{b}:\nabla\b{n}\b{f}_1\\
-\b{t}\b{b}\b{b}\b{t}:\nabla\b{n}P\b{f}_1-\b{b}\b{t}\b{t}\b{b}:\nabla\b{n}\b{f}_1
-(\b{t}\b{n}+\b{n}\b{t})\b{t}\cdot\nabla\b{f}_1
-(\b{b}\b{n}+\b{n}\b{b})\b{b}\cdot\nabla\b{f}_1.
\end{multline*}
If we introduce the forth order tensor $\Lambda=\b{t}\b{t}\b{t}\b{t}+
\b{b}\b{b}\b{b}\b{b}+\b{t}\b{b}\b{b}\b{t}+\b{b}\b{t}\b{t}\b{b}$ and the third order tensor $\Sigma=\b{t}\b{n}\b{t}+\b{n}\b{t}\b{t}+\b{b}\b{n}\b{b}
+\b{n}\b{b}\b{b}$, we then have
\begin{gather*}
\frac{1}{Re}[\nabla\nabla\b{u}]=-(\Lambda:\nabla\b{n})\b{f}_1
-\Sigma\cdot\nabla\b{f}_1+\b{n}\b{n}([\nabla p]-\kappa \b{f}_1).
\end{gather*}

Using the identity \eqref{eq:identity}, we find \begin{gather*}
\Lambda_{ijkl}=(\delta_{il}-n_in_l)(\delta_{jk}-n_jn_k),\\
\Sigma_{ijk}=(\delta_{ik}-n_in_k)n_j+n_i(\delta_{jk}-n_jn_k).
\end{gather*}
Simple computation reveals $$(\Lambda:\nabla\b{n})_{ij}=\partial_jn_i-n_k(\partial_kn_i)n_j,$$
and
$$
(\Sigma\cdot\nabla)_{ij}=n_j\partial_i+n_i\partial_j-2n_in_jn_k\partial_k.
$$
The claim is therefore true for 3D case. For 2D case, this general formula also holds and we omit the proof.

To simplify the general formula for the 2D case, we note $tr(\nabla\b{n})=\nabla\cdot\b{n}=-\kappa$ and $\b{t}\cdot\nabla\b{n}=-\kappa\b{t}$. Clearly, $(\nabla\b{n})\cdot\b{n}=0$, which implies that the components for $\b{t}\b{n},\b{n}\b{n}$ are all zero. Hence,
$\nabla\b{n}=\b{t}\b{t}(-\kappa)+\lambda\b{n}\b{t}$. That the extension of $\b{n}$ is undetermined implies that we can't determine $\lambda$. It follows that $$(I-\b{n}\b{n})\cdot\nabla\b{n}=-\kappa(I-\b{n}\b{n}).$$  This expression can also be obtained by noting $\nabla\b{n}^T-(\b{n}\b{n}\cdot\nabla\b{n})^T=\Lambda:\nabla\b{n}=\b{t}\b{t}(\b{t}\cdot(\b{t}\cdot\nabla\b{n}))=-\kappa\b{t}\b{t}$.
\end{proof}

We now turn our attention to the jump conditions of $\nabla\nabla p$. The results can be obtained similarly as those of $\nabla\nabla\b{u}$. 
\begin{thm}
With the setting in Section \ref{sec:eulerian}, the jump conditions of the second order derivatives of $p$ are given by 
\begin{multline}
[\partial_i\partial_jp]=n_in_j\Big\{-[\partial_lu_m\partial_mu_l]-\Delta f_2
+\kappa g\Big\}\\
+\partial_i\partial_j f_2+(\partial_jn_i-n_k\partial_kn_i n_j)g
+(n_j\partial_i+n_i\partial_j-2n_in_jn_k\partial_k)g,
\end{multline}
where $g=\nabla_{\Gamma}\cdot\b{f}_1-\b{n}\cdot\nabla f_2$. In the 2D case, denoting $u=u_1, v=u_2$, we have
\begin{multline}
[\partial_i\partial_jp]=n_in_j\Big\{2[\partial_xu\partial_y v]-2[\partial_xv\partial_y u]\Big\}+\partial_i\partial_j f_2-n_in_j\Delta f_2\\
-\kappa(\delta_{ij}-2n_in_j)g+(n_j\partial_i+n_i\partial_j-2n_in_jn_k\partial_k)g.
\end{multline}
\end{thm}

\begin{proof}
Taking the divergence of the momentum equation, we have
\begin{gather*}
[\Delta p]=-[\partial_i u_j\partial_j u_i].
\end{gather*}
Applying Lemma \ref{lmm:tangent} on $[\nabla p]$, we have $$
(I-\b{n}\b{n})\cdot[\nabla\nabla p]=(I-\b{n}\b{n})\cdot H,
$$
while the second order tensor $H$ is given by $$
H_{ij}=\partial_{ij}f_2+\partial_i n_j(\nabla_{\Gamma}\cdot\b{f}_1-\b{n}\cdot\nabla f_2)
+n_j\partial_i(\nabla_{\Gamma}\cdot\b{f}_1-\b{n}\cdot\nabla f_2)).
$$
To be convenient, we denote $g=\nabla_{\Gamma}\cdot\b{f}_1-\b{n}\cdot\nabla f_2$. Following a similar derivation as for $\nabla\nabla\b{u}$, we find
\begin{gather*}
[\partial_{tt}p]=t_it_jH_{ij}=t_it_j\partial_{ij}f_2+t_it_j\partial_i n_j g, \\
[\partial_{bb}p]=b_ib_j H_{ij}=b_ib_j\partial_{ij}f_2+b_ib_j\partial_{i}n_j g,\\
[\partial_{nn}p]=[\Delta p]-\Delta f_2+(\b{n}\b{n}:\nabla\nabla)f_2+\kappa g.
\end{gather*}
For the other components, 
\begin{gather*}
[\partial_{tn}p]=t_in_jH_{ij}=t_in_j\partial_{ij}f_2+t_i\partial_i g,\\
[\partial_{bn}p]=b_in_jH_{ij}=b_in_j\partial_{ij}f_2+b_i\partial_i g,\\
[\partial_{tb}p]=t_ib_j\partial_{ij}f_2+t_ib_j\partial_in_j g.
\end{gather*}
Hence, it follows that
\begin{gather*}
[\nabla\nabla p]-\b{n}\b{n}[\partial_{nn}p]
=[\nabla\nabla f_2-\b{n}\b{n}\b{n}\b{n}:\nabla\nabla f_2]
+\Lambda:\nabla \b{n} g+\Sigma\cdot\nabla g.
\end{gather*}
where $\Lambda$ and $\Sigma$ are defined the same as in the proof for the second derivatives of $\b{u}$.

In the 2D case, noting that $\partial_lu_m\partial_m u_l=-2\partial_xu\partial_yv+2\partial_xv\partial_yu$ by $\partial_xu+\partial_yv=0$, we then derived the claim.
\end{proof}

Now we have finished the derivation of spatial jump conditions. The expressions are in compact and clean forms, where the tangential forces are dealt with conveniently as well. They can be computed straightforwardly using the level set function $\phi$ and stretch function $\chi$.
\subsection{Temporal jumps}
When the interface $\Gamma$ passes the point $x$ at time $t$, a quantity $A$ is discontinuous in time. The temporal jump $[[A]]=A(x, t+)-A(x, t-)$ satisfies $[[A]]=[A]$ if $\b{u}\cdot\b{n}<0$, and $[[A]]=-[A]$ if $\b{u}\cdot\b{n}>0$. Here, $[\cdot]$ denotes the spatial jump defined before. In the application of the immersed interface method, one usually needs the particular temporal jumps: $[[\b{u}]]=0$ and $[[\b{u}_t]]$. $[\b{u}_t]=-\b{u}\cdot[\nabla\b{u}]$ can be computed easily. 

\section{Eulerian Immersed interface method} \label{sec:iim}
In this section, we propose a possible Eulerian immersed interface method for simulating the membranes or vesicles. The implementation will be left for future work. 
\subsection{The level set and stretch functions}
The level set function can be initially set to be the signed distance function to the interface. It can be initialized by Fast Sweeping Method \cite{zhao04} first and improved by solving the equation
\begin{gather}\label{eq:sgndis}
\varphi_{\tau}+\sgn(\phi_0)(|\nabla\varphi|-1)=0
\end{gather}
with subcell-resolution (see \cite{harten89, min10, lxs17}). Here, $\tau$ is a pseudo-time that relaxes the function to the signed distance function. 

The stretch function can be initialized using 
\begin{gather}\label{eq:stretch}
\chi_{\tau}+\sgn(\varphi)\nabla\varphi\cdot\nabla\chi=0.
\end{gather}
where $\varphi$ is the signed distance function obtained by Equation \eqref{eq:sgndis}.

The above two functions may be conveniently constructed in a tube containing the interface and the local level set method developed in \cite{pmozk99} then applies. As shown in \cite{lxs17}, $\chi$ may not be continuous globally and the local level set is actually desired. We'll then have to reinitialize them after certain steps. In the reinitialization step, we again solve these two equations to turn the level set function into a signed distance function and obtain a new stretch function. The analysis and numerical schemes for solving Equation \eqref{eq:stretch} were discussed in \cite{lxs17}. 

\subsection{The immersed interface method for Navier-Stokes equations}
With the spatial and temporal jump conditions derived in Section \ref{sec:jump}, one can then simulate the system with the immersed interface method under a fully Eulerian formulation. The issue is how to choose a suitable version of Navier-Stokes solvers with our jump conditions in hand.

Several versions of the immersed interface method for solving Navier-Stokes equations have been proposed in several papers. In \cite{lilai01, lkp06, ll03}, the combination of immersed interface method with projection methods was used, but the correction terms are quite complicated. In \cite{xw2d06}, an immersed interface method with direct discretization of the Navier-Stokes equations was proposed.  In \cite{beale15}, Beale considered an immersed interface method with a direct finite difference on a periodic domains. The convergence of the immersed interface method has been shown to be nearly second order with the assumption that the forces are given exactly. All seem to suggest the immersed interface method can be close to second order accurate, which is better than the immersed boundary method.

In 2D case, we propose to use the Navier-Stokes solver in \cite{xw2d06}, which has sufficient details for one to implement. The equations
\begin{gather}
\frac{\partial}{\partial t}\b{u}+\nabla\cdot(\b{u}\b{u})+\nabla p=\frac{1}{Re}\Delta \b{u}+\b{f}\\
\Delta p=-\frac{\partial D}{\partial t}-2\nabla\cdot(\b{u}D)+\frac{1}{Re}\Delta D+2(\partial_x u\partial_y v-\partial_y u\partial_x v)+\nabla\cdot\b{f},
\end{gather}
are proposed to solve, where $D=\nabla\cdot\b{u}$. $D$ is kept explicitly in the equation to enforce $D=0$ in numerics with MAC (marker-and-cell) grid.  Direct finite differences with correction terms have been proposed for these two equation in \cite{xw2d06}. The boundary condition for the pressure proposed there is a certain Neumann boundary condition. 

\appendix
\section{The force due to stretch} \label{app:stretchf}
In this section, we present the derivation of the force associated with the stretch energy.  The derivation here should be understood in the distribution sense as the Dirac delta function is involved and $\nabla \b{u}$ may be discontinuous (However, $\nabla\b{u}\cdot\b{n}$ and $\nabla\b{u}\cdot\nabla\phi$ are continuous). 

According to Equation \eqref{eq:elasf}, taking derivative on the stretch energy functional, and using the equations for $\chi,\phi$, we have
\begin{multline*}
\frac{d}{dt}\int_{\Omega}\frac{E_s(\chi)}{\chi}|\nabla\phi|\delta(\phi)dx
=\int_{\Omega} \left(\frac{E_s(\chi)}{\chi}\right)'\Big(-\b{u}\cdot\nabla\chi
-\b{n}\cdot\nabla\b{u}\cdot\b{n}\chi\Big)|\nabla\phi|\delta(\phi)dx\\
+\int_{\Omega} \frac{E_s(\chi)}{\chi}\delta'(\phi)(-\b{u}\cdot\nabla\phi)|\nabla\phi|dx
+\int_{\Omega}\frac{E_s(\chi)}{\chi}\delta(\phi)\frac{\nabla\phi}{|\nabla\phi|}\cdot\nabla(-\b{u}\cdot\nabla\phi)dx.
\end{multline*}
The prime associated with $E_s$ means derivative with respect to $\chi$ while the prime in $\delta$ means the derivative with respect to $\phi$ (the derivative of $\delta$ is understood in distribution sense). Integrating by parts, we have the following relations:
\begin{gather*}
\int_{\Omega} \left(\frac{E_s(\chi)}{\chi}\right)'(-\b{n}\cdot\nabla\b{u}\cdot\b{n}\chi)|\nabla\phi|\delta(\phi)dx
=\int_{\Omega}\b{u}\cdot\nabla\cdot\left(\Big(\frac{E_s(\chi)}{\chi}\Big)'\chi\b{n}\b{n}|\nabla\phi|\delta(\phi)\right)dx.\\
\int_{\Omega}\frac{E_s(\chi)}{\chi}\delta(\phi)\frac{\nabla\phi}{|\nabla\phi|}\cdot\nabla(-\b{u}\cdot\nabla\phi)dx
=\int_{\Omega}\b{u}\cdot\nabla\phi\nabla\cdot\left(\frac{E_s(\chi)}{\chi}\delta(\phi)\frac{\nabla\phi}{|\nabla\phi|}\right)dx.
\end{gather*}
Since $\b{u}$ is divergence free, we obtain the force due to stretch up to a total gradient:
\begin{multline*}
\b{f}_s=\nabla\left(\frac{E_s(\chi)}{\chi}\right)|\nabla\phi|\delta(\phi)
-\nabla\cdot\left(\Big(\frac{E_s(\chi)}{\chi}\Big)'\chi\b{n}\b{n}|\nabla\phi|\delta(\phi)\right)\\
+\frac{E_s(\chi)}{\chi}|\nabla\phi|\nabla\delta(\phi)-\nabla\phi\nabla\cdot\left(\frac{E_s(\chi)}{\chi}\delta(\phi)\frac{\nabla\phi}{|\nabla\phi|}\right)+\nabla p'.
\end{multline*}
Since $(E_s/\chi)'\chi=-E_s/\chi+E_s'$ and
\begin{gather*}
\nabla\cdot\left(\frac{E_s(\chi)}{\chi}\b{n}\b{n}|\nabla\phi|\delta(\phi)\right)-\nabla\phi\nabla\cdot\left(\frac{E_s(\chi)}{\chi}\delta(\phi)\frac{\nabla\phi}{|\nabla\phi|}\right)
=\frac{E_s(\chi)}{\chi}\delta(\phi)\b{n}\cdot\nabla\nabla\phi=\frac{E_s(\chi)}{\chi}\delta(\phi)\nabla|\phi|,
\end{gather*}
by absorbing a total gradient $\nabla(\frac{E_s(\chi)}{\chi}|\nabla\phi|\delta(\phi))$ into $\nabla p'$, we get
\begin{gather*}
\b{f}_s=-\nabla\cdot(E_s' \b{n}\b{n}|\nabla\phi|\delta(\phi))+\nabla\tilde{p}
=-\nabla\cdot(E_s' \b{n}\nabla\phi\delta(\phi))+\nabla\tilde{p}.
\end{gather*}
By the product rule
\begin{gather*}
-\nabla\cdot(E_s' \b{n}\nabla\phi\delta(\phi))
=-\nabla E_s'\cdot \b{n}\b{n}|\nabla\phi|\delta(\phi)
+E_s'\kappa \nabla\phi\delta(\phi)-E_s'\b{n}\cdot\nabla\nabla\phi-E_s'\nabla\phi \b{n}\cdot\nabla\delta(\phi),
\end{gather*}
that $\nabla\phi\b{n}\cdot\nabla\delta(\phi)=|\nabla\phi|\nabla\delta(\phi)$
and that $n\cdot\nabla\nabla\phi=\nabla|\phi|$, we can again absorb a total gradient $-\nabla(E_s'|\nabla\phi|\delta(\phi))$ into $\nabla\tilde{p}$. Finally, we obtain the expression 
\begin{gather}
\b{f}_s=-\nabla E_s'\cdot \b{n}\b{n}|\nabla\phi|\delta(\phi)+E_s'\kappa \nabla\phi\delta(\phi)+(\nabla E_s')|\nabla\phi|\delta(\phi)+\nabla p^*.
\end{gather}
Since $p^*$ can be combined with the pressure, we then have obtained the desired result.

\section{The force due to curvature}\label{app:curvf}
If possible, one may understand the derivation here in the distribution sense as one can take derivatives of discontinuous functions in the distribution sense. 

Note that 
\begin{gather}
(I-\b{n}\b{n})\cdot\nabla|\nabla\phi|=(n\cdot\nabla \frac{\nabla\phi}{|\nabla\phi|})|\nabla\phi|
\end{gather}
Using the identity $\kappa=-\nabla\cdot\b{n}$, we reduce $\b{f}_b$ to 
\begin{multline}
\b{f}_b=\nabla\cdot(-E_b\b{n}-(I-\b{n}\b{n})\cdot\nabla E_b'
-E_b'\b{n}\cdot\nabla\b{n})\nabla\phi \delta(\phi)\\
=(E_b\kappa -\Delta E_b'-\kappa\b{n}\cdot\nabla E_b'
+n_in_j\partial^2_{ij}E_b'-E_b'\partial_jn_i\partial_in_j)\nabla\phi \delta(\phi)
\end{multline}
Finally, for a given function $g$, we have 
\begin{multline}
(\nabla-\b{n}\b{n}\cdot\nabla)\cdot(\nabla g-\b{n}\b{n}\cdot\nabla g)
=\Delta g+\kappa\b{n}\cdot\nabla g-\b{n}\cdot\nabla(\b{n}\cdot\nabla g)
-\b{n}\b{n}:\nabla\nabla g\\
+(\b{n}\cdot\nabla\nabla\b{n}\cdot\b{n})\b{n}\cdot\nabla g
+(\b{n}\cdot\b{n})\b{n}\cdot\nabla(\b{n}\cdot\nabla g)
\end{multline}
Hence, the force is reduced to 
\begin{gather}
\b{f}_b=(E_b\kappa-\Delta_{tan}E_b'-E_b'\nabla\b{n}:\nabla\b{n}^T)\nabla\phi \delta(\phi)
\end{gather}
\bibliographystyle{mybst}
\bibliography{NumAnaPde}
\end{document}